\documentclass[11pt]{article}
\usepackage[tbtags]{amsmath}
\usepackage{amssymb}
\usepackage{amsthm}
\usepackage[misc]{ifsym}
\usepackage{cases}
\usepackage{mathrsfs}
\usepackage{graphicx}
\usepackage{float}
\usepackage[colorlinks,
linkcolor=blue,
anchorcolor=blue,
citecolor=blue]{hyperref}
\usepackage{relsize}
\numberwithin{equation}{section}   
\normalsize
\title{\bf A Verification Theorem for Stackelberg Stochastic Differential Games in Feedback Information Pattern\thanks{This work is financially supported by National Key R\&D Program of China (Grant No. 2018YFB1305400), National Natural Science Foundations of China (Grant Nos. 11971266, 11831010, 11571205), and Shandong Provincial Natural Science Foundations (Grant Nos. ZR2020ZD24, ZR2019ZD42).}}
\author{\normalsize Qi Huang\thanks{\it School of Mathematics, Shandong University, Jinan 250100, P.R. China, E-mail: 201911814@mail.sdu.edu.cn},\quad Jingtao Shi\thanks{Corresponding author, \it School of Mathematics, Shandong University, Jinan 250100, P.R. China, E-mail: shijingtao@sdu.edu.cn}}

\newtheorem{mythm}{Theorem}[section]
\newtheorem{mydef}{Definition}[section]

\newtheorem{myrem}{Remark}[section]

\begin{document}
\maketitle

\noindent{\bf Abstract:}\quad This paper is concerned with a Stackelberg stochastic differential game on a finite horizon in feedback information pattern. A system of parabolic partial differential equations is obtained at the level of Hamiltonian to give the verification theorem of the feedback Stackelberg equilibrium. As an example, a linear quadratic Stackelberg stochastic differential game is investigated. Riccati equations are introduced to express the feedback Stackelberg equilibrium, analytical and numerical solutions to these Riccati equations are discussed in some special cases.

\vspace{2mm}


\vspace{2mm}

\section{Introduction}

\indent The Stackelberg solution concept was first introduced by Von Stackelberg in \cite{St1934}. This solution concept arises in two-person nonzero-sum static games with asymmetrical modes of play. One of the players is called the leader, and the other is called the follower. The leader has the ability to announce his policy first, leaving to the follower to react. The follower optimizes his cost functional given the strategy that the leader has announced. Anticipating the follower's response, the leader choose the policy which will minimize his cost functional under the follower's rational response.

\indent The equilibrium concept was first extended to open-loop Stackelberg equilibrium for dynamic games. The leader announces the policy which he is going to take for the rest of the game. And taking this policy as given, the follower choose his policy to minimize his cost functional. However, as we all know, this equilibrium can be time inconsistent. So the focus of research shifted to finding feedback Stackelberg equilibrium which was first introduced by Simman and Cruz in multi-period games in \cite{SC1973-1,SC1973-2}. In this kind of equilibrium, the leader merely has a stagewise first-mover advantage over the follower. The players' feedback strategies depend on the observed pair (date, state), such that at any (date, state) pair, the continuation of optimal strategy remains optimal for the players.

\indent The feedback Stackelberg equilibrium was originally defined by Ba\c{s}ar and Hauire in \cite{BH1984}. The leader has a first-mover advantage over the follower at every stage of the game, which means that the leader has an instantaneous advantage at every point in time. As shown in \cite{BH1984}, the continuous-time problem can be regarded as the limit of a series of discrete-time games which is the set of pointwise Stackelberg solutions to coupled {\it Hamilton-Jacobi-Bellman} (HJB, in short) equations. To compute the feedback Stackelberg equilibrium, the follower's pointwise best response to the leader's policy is computed firstly. Secondly, according to the follower's pointwise best response, the leader solves his optimization problem. Applying backward induction, we need to substitute the follower's instantaneous reaction function into the leader's HJB equation and find the leader's optimal feedback strategy by maximizing the right-hand side of the equation. For this equilibrium, there is no need to assume any responsibility over the entire time horizon, only a periodic first-mover advantage. This equilibrium is subgame perfect and time consistent. On the other hand, if the leader announces his policy for the rest of the game at the initial time, the follower minimizes his cost functional under this strategy. The leader has global advantage over the follower. Derivation of global Stackelberg solution is still an active area of research (Mart{\'\i}n-Herr{\'a}n and Rubio \cite{martin2021coincidence}). See Ba\c{s}ar and Olsder \cite{BO1998} for more types of Stackelberg solutions and their connections, and see He et al. \cite{HPS2009}, Chen and Shen \cite{CS2018}, Shi et al. \cite{SWX2016}, Zheng and Shi \cite{zheng2021stackelberg} for some recent progress in Stackelberg stochastic differential games.

\indent \cite{BH1984} derives a coupled system of HJB equations which are parabolic {\it partial differential equations} (PDEs, in short), to characterize the feedback Stackelberg equilibrium. As a special case, they consider a {\it linear-quadratic} (LQ, in short) game, deduce the associated Riccati equation, and give the existence of the solution to it within a sufficiently small horizon. The uncertainty in \cite{BH1984} comes from a finite-state stochastic jump process. In contrast to \cite{BH1984}, Bensoussan, Chen and Sethi in \cite{BCS2014} consider an infinite-horizon Stackelberg stochastic differential game in which the uncertainty comes from a standard Brownian motion. They obtain a sufficient condition for the feedback Stackelberg equilibrium and apply it to the LQ case. Due to the infinite horizon nature, the HJB equations in this case are elliptic PDEs.

\indent In this paper, we consider a finite-horizon Stackelberg stochastic differential game involving a Brownian motion. Different from \cite{BCS2014}, both the drift term and the diffusion term of the state equation in this paper contain the leader's and the follower's control variables. A verification theorem in feedback information pattern is first obtained. We use a system of parabolic PDEs obtained from the Stackelberg game at the level of Hamiltonian to get the sufficient condition for the feedback Stackelberg equilibrium. Compared with \cite{BCS2014}, the Hamiltonian functions of the leader and the follower in our case become more complex. And we apply the verification theorem to the special LQ case. In this case, the state equation is as follows:
\begin{equation*}
\left\{
\begin{aligned}
dx(s)&=\big[A(s)x(s)+B_1(s)u(s)+B_2(s)v(s)+b(s)\big]ds\\
     &\quad+\big[C(s)x(s)+D_1(s)u(s)+D_2(s)v(s)+\lambda(s)\big]dW(s),\quad s \in [t,T],\\
 x(t)&=x,
\end{aligned}
\right.
\end{equation*}
where $(t,x)$ is the initial time and state pair, $A(\cdot), B_1(\cdot), B_2(\cdot), C(\cdot), D_1(\cdot), D_2(\cdot), b(\cdot), \lambda(\cdot)$ are given matrix-valued deterministic functions, and $W(\cdot)$ is a one-dimensional Brownian motion. And the cost functionals are:
\begin{equation*}
\begin{aligned}
&J_i(t,x;u(\cdot),v(\cdot))=\frac{1}{2}\mathbb E\Bigg\{\mathlarger{\int^T_t} \left[\left\langle\left( \begin{array}{ccc}
Q_i(s) & {M_{i1}(s)}^\top & {M_{i2}(s)}^\top\\
M_{i1}(s) & R_{i11}(s) & R_{i12}(s)\\
M_{i2}(s) & R_{i21}(s) & R_{i22}(s)\\
\end{array} \right)
\left( \begin{array}{ccc}
x(s)\\
u(s)\\
v(s)\\
\end{array} \right),
\left( \begin{array}{ccc}
x(s)\\
u(s)\\
v(s)\\
\end{array} \right)\right\rangle\right.\\
&\qquad\qquad+2\left.\left\langle \left( \begin{array}{ccc}
q_i(s)\\
\rho_{i1}(s)\\
\rho_{i2}(s)\\
\end{array} \right),
\left( \begin{array}{ccc}
x(s)\\
u(s)\\
v(s)\\
\end{array} \right)
\right\rangle
\right]ds+\langle L_i x(T), x(T)\rangle+2\langle N_i, x(T)\rangle\Bigg\},
\end{aligned}
\end{equation*}
where $Q_i(\cdot), M_{i1}(\cdot), M_{i2}(\cdot), R_{i11}(\cdot), R_{i12}(\cdot), R_{i21}(\cdot), R_{i22}(\cdot), q_i(\cdot), \rho_{i1}(\cdot), \rho_{i2}(\cdot)$ are given matrix-\\valued deterministic functions, $L_i$ is an $n \times n$ symmetric matrix and $N_i \in \mathbb{R}^n$, for $i=1,2$.

Noting that the diffusion term of the state equation contains the leader's and the follower's control variables, and the form of cost functionals of the leader and the follower are very general. We derive the corresponding system of HJB equations in the verification theorem. We then consider two special LQ cases to get the corresponding representations of the feedback Stackelberg equilibrium, via some Riccati equations. The solvability of them is discussed as well.

\indent The rest of this paper is organized as follows. Section \ref{Section 2} gives the formulation of the Stackelberg stochastic differential game and give the definition of feedback Stackelberg equilibrium. Section \ref{A} is devoted to the verification theorem of the feedback Stackelberg equilibrium. In Section \ref{Section 4}, an LQ case is researched. We use Riccati equations to express the feedback Stackelberg equilibrium, and discuss the analytical and numerical solutions to them in some special cases. Finally, in Section \ref{Section 5}, some concluding remarks are given.

\section{Problem Formulation}\label{Section 2}

\noindent For given $T>0$ and initial data $(t,x)\in [0,T]\times \mathbb R^n$, a Stackelberg stochastic differential game is considered. The state equation is
\begin{equation}\label{state SDE}
\left\{
\begin{aligned}
dx(s)&=f(s,x(s),u(s),v(s))ds+\sigma(s,x(s),u(s),v(s))dW(s),\quad 0\le t \le s \le T, \\
 x(t)&=x,
\end{aligned}
\right.
\end{equation}
where $(\Omega,\mathcal{F},\{\mathcal{F}_t\}_{t\ge 0},\mathbb{P})$ is a filtered probability space, and $W(\cdot)$ is a $d$-dimensional standard Brownian motion defined on it.
$f:[t,T]\times \mathbb{R}^n \times \mathbb{R}^{m_1}\times \mathbb{R}^{m_2} \to \mathbb{R}^n$ and $\sigma :[t,T]\times \mathbb{R}^n \times \mathbb{R}^{m_1}\times \mathbb{R}^{m_2} \to \mathbb{R}^{n\times d}$
are measurable functions. There exists a constant $C>0$ such that
\[|f(s,x,u,v)-f(s,x',u',v')|+|\sigma(s,x,u,v)-\sigma(s,x',u',v')| \le C[|x-x'|+|u-u'|+|v-v'|],\]
\[\forall s \in [t,T], x, x' \in \mathbb{R}^n, u, u' \in \mathbb{R}^{m_1}, v, v' \in \mathbb{R}^{m_2}.\]
And for some $p \in [2, \infty)$, $\left(\int_t^T |f(s,0,0,0)| ds\right)^p+\left(\int_t^T |\sigma(s,0,0,0)|^2 ds\right)^{\frac{p}{2}} < \infty$.

$u(\cdot)$ and $v(\cdot)$ are control processes of two players: the leader (player 1) and the follower (player 2), respectively.
The cost functionals for the leader and the follower are of the form:
\begin{equation}\label{cost functional-leader}
\begin{aligned}
J_1(t,x;u(\cdot),v(\cdot))=\mathbb E\left[\int^T_t g_1(s,x(s),u(s),v(s))ds+h_1(x(T))\right],
\end{aligned}
\end{equation}
\begin{equation}\label{cost functional-follower}
\begin{aligned}
J_2(t,x;u(\cdot),v(\cdot))=\mathbb E\left[\int^T_t g_2(s,x(s),u(s),v(s))ds+h_2(x(T))\right],
\end{aligned}
\end{equation}
where $g_i:[t,T]\times \mathbb{R}^n \times \mathbb{R}^{m_1}\times \mathbb{R}^{m_2} \to \mathbb{R}$ and $h_i: \mathbb{R}^n \to \mathbb{R}$ ($i=1,2$)
are measurable functions which satisfy the following polynomial growth conditions: For $i=1,2$,
\begin{equation*}
\begin{aligned}
&|g_i(s,x,u,v)|\le C(1+|x|^p+|u|^p+|v|^p),\quad |h_i(x)|\le C(1+|x|^p),\\
&\hspace{2cm} \quad \forall (s,x,u,v)\in [t,T]\times \mathbb{R}^n \times \mathbb{R}^{m_1}\times \mathbb{R}^{m_2},
\end{aligned}
\end{equation*}
for some positive constants $C$ and $p$ mentioned above.

In our Stackelberg stochastic differential game in feedback information pattern, the leader determines his instantaneous strategy of the form $u(s,x(\cdot)), s\in[t,T]$. And according to the observed state $x(\cdot)$ and the leader's instantaneous strategy as the game progress, the follower makes his instantaneous decision $v(s,x(\cdot),u(s,x(\cdot))), s\in[t,T]$. So the admissible strategy spaces for the leader and the follower are as follows:
\begin{equation*}
\mathcal{U}[0,T]=\Big\{u(\cdot,\cdot)\in L^p_{\mathcal{F}} (t,T;{\mathbb{R}}^{m_1})\big|u:[t,T]\times {\mathbb{R}}^n \to U, u(s,x) \textrm{ is Lipschitz continuous in } (s,x)\Big\},
\end{equation*}
\begin{equation*}
\begin{aligned}
\mathcal{V}[0,T]&=\Big\{v(\cdot,\cdot,\cdot)\in L^p_{\mathcal{F}} (t,T;{\mathbb{R}}^{m_2})\big|v:[t,T]\times {\mathbb{R}}^n \times U \to V, v(s,x,u)\textrm{ is Lipschitz continuous}\\
                &\qquad\qquad\qquad\qquad\qquad\qquad\quad\textrm{in }(s,x,u)\Big\},
\end{aligned}
\end{equation*}
where $U$ and $V$ are given subsets in ${\mathbb{R}}^{m_1}$ and ${\mathbb{R}}^{m_2},$ respectively. And $L^p_{\mathcal{F}} (t,T;{\mathbb{R}}^{k})$ are the set of all $\{{\mathcal{F}}_t\}_{t\ge 0}$- adapted ${\mathbb{R}}^{k}$-valued processes $X(\cdot)$ such that $\mathbb{E} \int_t^T |X(s)|^p ds < \infty$, $k=m_1, m_2.$

For a pair of strategies $(u(\cdot,\cdot),v(\cdot,\cdot,u(\cdot,\cdot)))\in \mathcal{U}[0,T]\times \mathcal{V}[0,T]$, we use $x^{t,x}(\cdot;u,v)$ to denote the solution to the parameterized state equation
\begin{equation}\label{parameterized state SDE}
\left\{
\begin{aligned}
dx(s)&=f\big(s,x(s),u(s,x(s)),v(s,x(s),u(s,x(s)))\big)ds\\
     &\quad+\sigma\big(s,x(s),u(s,x(s)),v(s,x(s),u(s,x(s)))\big)dW(s),\quad 0\le t \le s \le T, \\
 x(t)&=x.
\end{aligned}
\right.
\end{equation}
And for $i=1,2$, we use $J^{t,x}_i (u(\cdot,\cdot),v(\cdot,\cdot,u(\cdot,\cdot)))$ to represent the corresponding cost functional of player $i$:
\begin{equation}\label{parameterized cost functionals}
\begin{split}
&J^{t,x}_i (u(\cdot,\cdot),v(\cdot,\cdot,u(\cdot,\cdot)))\\
&=\mathbb E\bigg\{\int^T_t g_i\big(s,x^{t,x}(s;u,v),u(s,x^{t,x}(s;u,v)),v(s,x^{t,x}(s;u,v),u(s,x^{t,x}(s;u,v)))\big)ds\\
&\qquad\quad+h_i(x^{t,x}(T;u,v))\bigg\}.
\end{split}
\end{equation}

\begin{mydef}
If the following holds:
\begin{equation}\label{feedback Stackelberg equilibrium}
\left\{
\begin{aligned}
&J^{t,x}_1 (u^*(\cdot,\cdot),v^*(\cdot,\cdot,u^*(\cdot,\cdot)))\le J^{t,x}_1 (u(\cdot,\cdot),v^*(\cdot,\cdot,u(\cdot,\cdot))),\\
&\hspace{4cm}\forall u(\cdot,\cdot)\in \mathcal{U}[0,T],\quad \forall (t,x)\in [0,T]\times \mathbb{R}^n,\\
&J^{t,x}_2 (u^*(\cdot,\cdot),v^*(\cdot,\cdot,u^*(\cdot,\cdot)))\le J^{t,x}_2 (u^*(\cdot,\cdot),v(\cdot,\cdot,u^*(\cdot,\cdot))),\\
&\hspace{4cm}\forall v(\cdot,\cdot,\cdot)\in \mathcal{V}[0,T],\quad \forall (t,x)\in [0,T]\times \mathbb{R}^n,
\end{aligned}
\right.
\end{equation}
\noindent we call the pair of strategies $(u^*(\cdot,\cdot),v^*(\cdot,\cdot,u^*(\cdot,\cdot)))\in \mathcal{U}[0,T]\times \mathcal{V}[0,T]$ is a feedback Stackelberg equilibrium.
\end{mydef}
\begin{myrem}
This definition seems very similar to the definition of the feedback Nash equilibrium. In fact, they are different. Because in the feedback Stackelberg equilibrium, the strategy of the follower is influenced by the leader's strategy. However, in the feedback Nash equilibrium, each player has equal roles and status. On the other hand, in the feedback Stackelberg equilibrium, the leader merely has an instantaneous advantage over the follower at every point in time, and players' feedback strategies depend on the observed time and state. Therefore, there is no need to assume any responsibility over the entire time horizon. This equilibrium is subgame perfect and time consistent. See more detail in \cite{BCS2014}.
\end{myrem}

\section{Verification Theorem}\label{A}

In this section, we will give a verification theorem, which provides a sufficient condition for the feedback Stackelberg equilibrium.
Let $\mathcal{S}^n$ denote the set of symmetric $n \times n$ matrices $A=(A_{ij}), i,j=1, \cdots, n.$ Let $a=\sigma {\sigma}^\top$ and $\mathrm{tr}[aA]=\sum\limits_{i,j=1}^n a_{ij}A_{ij}.$
We introduce the Hamiltonian functions for the leader and the follower as follows:
\[H_1(s,x,\mu,\nu,p_1,A'):=\langle p_1,f(s,x,\mu,\nu)\rangle+\frac{1}{2}\mathrm{tr}[a(s,x,\mu,\nu)A']+g_1(s,x,\mu,\nu),\]
\[H_2(s,x,\mu,\nu,p_2,A''):=\langle p_2,f(s,x,\mu,\nu)\rangle+\frac{1}{2}\mathrm{tr}[a(s,x,\mu,\nu)A'']+g_2(s,x,\mu,\nu),\]
where $H_i:[t,T]\times \mathbb{R}^n \times \mathbb{R}^{m_1}\times \mathbb{R}^{m_2}\times \mathbb{R}^n \times \mathcal{S}^n \to \mathbb{R}$, for $i=1,2$.

For every $(s,x,\mu,p_2,A'')$, suppose $H_2$ is strictly convex in $\nu$.
Therefore, the follower has a unique optimal response function for the leader's each policy $\mu \in U$:
\[T_2(s,x,\mu,p_2,A''):=\mathop{argmin}\limits_{\nu \in\, V} H_2(s,x,\mu,\nu,p_2,A'').\]
Under the follower's optimal response $T_2$, the leader should take a strategy to minimize his Hamiltonian functional $H_1(s,x,\mu,T_2(s,x,\mu,p_2,A''),p_1,A')$. We further assume that it is also strictly convex in $\mu$, for every $(s,x,p_1,p_2,A',A'')$. So the leader's optimal action is
\[T_1(s,x,p_1,p_2,A',A''):=\mathop{argmin}\limits_{\mu \in\, U} H_1(s,x,\mu,T_2(s,x,\mu,p_2,A''),p_1,A').\]
Then we obtain a feedback Stackelberg equilibrium
\[\big(T_1(s,x,p_1,p_2,A',A''), T_2(s,x,T_1(s,x,p_1,p_2,A',A''),p_2,A'')\big).\]

Let $C_p([t,T]\times \mathbb{R}^n)$ denote the set of all continuous functions $\Phi(s,x)$ on $[t,T]\times \mathbb{R}^n$ satisfying
a polynomial growth condition
\[|\Phi (s,x)| \le C(1+|x|^p)\]
for some positive constant $C$ and $p$ mentioned above. Let $C^{1,2}([t,T]\times \mathbb{R}^n)$ denote the set of all continuous functions $\Phi(s,x)$ on $[t,T]\times \mathbb{R}^n$ with continuous partial derivative in $s$ and 2-order continuous derivative in $x$.

With these notations, we have the following verification theorem.
\begin{mythm}
Suppose $V_1(s,x)$, $V_2(s,x)$ both lie in $C^{1,2}([t,T]\times \mathbb{R}^n) \cap C_p([t,T]\times \mathbb{R}^n)$ and solve the system of parabolic PDEs
\begin{equation}\label{PDE-1}
\left\{
\begin{aligned}
&\frac{\partial V_1}{\partial s}(s,x)+\bigg\langle\frac{\partial V_1}{\partial x}(s,x),f\bigg(s,x,T_1\Big(s,x,\frac{\partial V_1}{\partial x},\frac{\partial V_2}{\partial x},\frac{{\partial}^2 V_1}{\partial x^2},\frac{{\partial}^2 V_2}{\partial x^2}\Big),\\
&\qquad T_2\Big(s,x,T_1\Big(s,x,\frac{\partial V_1}{\partial x},\frac{\partial V_2}{\partial x},\frac{{\partial}^2 V_1}{\partial x^2},\frac{{\partial}^2 V_2}{\partial x^2}\Big),\frac{\partial V_2}{\partial x},\frac{{\partial}^2 V_2}{\partial x^2}\Big)\bigg)\bigg\rangle\\
&+\frac{1}{2} \sum\limits_{i,j=1}^n \Bigg[a_{ij}\bigg(s,x,T_1\Big(s,x,\frac{\partial V_1}{\partial x},\frac{\partial V_2}{\partial x},\frac{{\partial}^2 V_1}{\partial x^2},\frac{{\partial}^2 V_2}{\partial x^2}\Big),\\
&\qquad T_2\Big(s,x,T_1\Big(s,x,\frac{\partial V_1}{\partial x},\frac{\partial V_2}{\partial x},\frac{{\partial}^2 V_1}{\partial x^2},\frac{{\partial}^2 V_2}{\partial x^2}\Big),\frac{\partial V_2}{\partial x},\frac{{\partial}^2 V_2}{\partial x^2}\Big)\bigg)\cdot \frac{{\partial}^2 V_1}{\partial x_i \partial x_j}(s,x)\Bigg]\\
&+g_1\bigg(s,x,T_1\Big(s,x,\frac{\partial V_1}{\partial x},\frac{\partial V_2}{\partial x},\frac{{\partial}^2 V_1}{\partial x^2},\frac{{\partial}^2 V_2}{\partial x^2}\Big),\\
&\qquad T_2\Big(s,x,T_1\Big(s,x,\frac{\partial V_1}{\partial x},\frac{\partial V_2}{\partial x},\frac{{\partial}^2 V_1}{\partial x^2},\frac{{\partial}^2 V_2}{\partial x^2}\Big),\frac{\partial V_2}{\partial x},\frac{{\partial}^2 V_2}{\partial x^2}\Big)\bigg)=0, \quad s \in [t,T],\\
&V_1(T,x)=h_1(x),
\end{aligned}
\right.
\end{equation}
\begin{equation}\label{PDE-2}
\left\{
\begin{aligned}
&\frac{\partial V_2}{\partial s}(s,x)+\bigg\langle\frac{\partial V_2}{\partial x}(s,x),f\bigg(s,x,T_1\Big(s,x,\frac{\partial V_1}{\partial x},\frac{\partial V_2}{\partial x},\frac{{\partial}^2 V_1}{\partial x^2},\frac{{\partial}^2 V_2}{\partial x^2}\Big),\\
&\qquad T_2\Big(s,x,T_1\Big(s,x,\frac{\partial V_1}{\partial x},\frac{\partial V_2}{\partial x},\frac{{\partial}^2 V_1}{\partial x^2},\frac{{\partial}^2 V_2}{\partial x^2}\Big),\frac{\partial V_2}{\partial x},\frac{{\partial}^2 V_2}{\partial x^2}\Big)\bigg)\bigg\rangle\\
&+\frac{1}{2} \sum\limits_{i,j=1}^n \Bigg[a_{ij}\bigg(s,x,T_1\Big(s,x,\frac{\partial V_1}{\partial x},\frac{\partial V_2}{\partial x},\frac{{\partial}^2 V_1}{\partial x^2},\frac{{\partial}^2 V_2}{\partial x^2}\Big),\\
&\qquad T_2\Big(s,x,T_1\Big(s,x,\frac{\partial V_1}{\partial x},\frac{\partial V_2}{\partial x},\frac{{\partial}^2 V_1}{\partial x^2},\frac{{\partial}^2 V_2}{\partial x^2}\Big),\frac{\partial V_2}{\partial x},\frac{{\partial}^2 V_2}{\partial x^2}\Big)\bigg)\cdot \frac{{\partial}^2 V_2}{\partial x_i \partial x_j}(s,x)\Bigg]\\
&+g_2\bigg(s,x,T_1\Big(s,x,\frac{\partial V_1}{\partial x},\frac{\partial V_2}{\partial x},\frac{{\partial}^2 V_1}{\partial x^2},\frac{{\partial}^2 V_2}{\partial x^2}\Big),\\
&\qquad T_2\Big(s,x,T_1\Big(s,x,\frac{\partial V_1}{\partial x},\frac{\partial V_2}{\partial x},\frac{{\partial}^2 V_1}{\partial x^2},\frac{{\partial}^2 V_2}{\partial x^2}\Big),\frac{\partial V_2}{\partial x},\frac{{\partial}^2 V_2}{\partial x^2}\Big)\bigg)=0, \quad s \in [t,T],\\
&V_2(T,x)=h_2(x),
\end{aligned}
\right.
\end{equation}
where $a_{ij}:=\sum\limits_{k=1}^d \sigma_{ik}\sigma_{jk}$. If we set
\[u^*(s,x):=T_1\Big(s,x,\frac{\partial V_1}{\partial x},\frac{\partial V_2}{\partial x},\frac{{\partial}^2 V_1}{\partial x^2},\frac{{\partial}^2 V_2}{\partial x^2}\Big) \quad \textrm{and} \quad v^*(s,x,\mu):=T_2\Big(s,x,\mu,\frac{\partial V_2}{\partial x},\frac{{\partial}^2 V_2}{\partial x^2}\Big),\]
then $(u^*(\cdot,\cdot),v^*(\cdot,\cdot,u^*(\cdot,\cdot)))$ is a feedback Stackelberg equilibrium.
\end{mythm}

\begin{proof}
Suppose the leader adopts the strategy $u^*(\cdot,\cdot)$ and the follower chooses an arbitrary strategy $v(\cdot,\cdot,\cdot)\in \mathcal{V}[0,T]$.
Let $B_n(x)$ denotes the open ball of radius $n$, centered at $x$, i.e.,
\[B_n(x):=\left\{y\in\mathbb{R}^n|\sqrt{\sum\limits_{i=1}^n |x_i-y_i|^2}<n\right\},\]
and let $\tau_n$ be the first exit time of $x^{t,x}(\cdot;u^*,v)$ from $B_n(x)$, i.e.,
\[\tau_n:=\inf\big\{s|x^{t,x}(s;u^*,v)\notin B_n(x),\ t\le s\le T\big\}.\]
If $x^{t,x}(s;u^*,v)\in B_n(x)$ for all $s\in[t,T]$, then $\tau_n=T$.

Applying It\^{o}'s formula to $V_2(\cdot,x^{t,x}(\cdot;u^*,v))$, integrating from $t$ to $\tau_n$, and taking expectation, we obtain
\begin{equation*}
\begin{split}
&V_2(t,x)=\mathbb E\big[V_2(\tau_n,x^{t,x}(\tau_n;u^*,v))\big]\\
&\quad-\mathbb E\Bigg\{\int_t^{\tau_n}\bigg[\frac{\partial V_2}{\partial s}(s,x^{t,x}(s;u^*,v))+\bigg\langle\frac{\partial V_2}{\partial x}(s,x^{t,x}(s;u^*,v)),f\big(s,x^{t,x}(s;u^*,v),\\
&\qquad\qquad\qquad u^*(s,x^{t,x}(s;u^*,v)),v(s,x^{t,x}(s;u^*,v),u^*(s,x^{t,x}(s;u^*,v)))\big)\bigg\rangle\\
&\qquad\qquad+\frac{1}{2} \sum\limits_{i,j=1}^n \Big[a_{ij}\big(s,x^{t,x}(s;u^*,v),u^*(s,x^{t,x}(s;u^*,v)),\\
&\qquad\qquad\qquad v(s,x^{t,x}(s;u^*,v),u^*(s,x^{t,x}(s;u^*,v)))\big)\cdot\frac{{\partial}^2 V_2}{\partial x_i \partial x_j}(s,x^{t,x}(s;u^*,v))\Big]\\
&\qquad\qquad+g_2\big(s,x^{t,x}(s;u^*,v),u^*(s,x^{t,x}(s;u^*,v)),v(s,x^{t,x}(s;u^*,v),u^*(s,x^{t,x}(s;u^*,v)))\big)\bigg] ds \Bigg\}\\
&\quad+\mathbb
E\int_t^{\tau_n}g_2\big(s,x^{t,x}(s;u^*,v),u^*(s,x^{t,x}(s;u^*,v)),v(s,x^{t,x}(s;u^*,v),u^*(s,x^{t,x}(s;u^*,v)))\big)ds.
\end{split}
\end{equation*}
Let $n\to \infty$ and noting $\tau_n \to T$ almost surely, we get
\begin{equation*}
\begin{split}
&V_2(t,x)=\mathbb E\big[V_2(T,x^{t,x}(T;u^*,v))\big]\\
&\quad+\mathbb E\int_t^{T}g_2\big(s,x^{t,x}(s;u^*,v),u^*(s,x^{t,x}(s;u^*,v)),v(s,x^{t,x}(s;u^*,v),u^*(s,x^{t,x}(s;u^*,v)))\big)ds\\
&\quad-\mathbb E\Bigg\{\int_t^{T}\bigg[\frac{\partial V_2}{\partial s}(s,x^{t,x}(s;u^*,v))+\bigg\langle\frac{\partial V_2}{\partial x}(s,x^{t,x}(s;u^*,v)),f\big(s,x^{t,x}(s;u^*,v),\\
&\qquad\qquad\qquad u^*(s,x^{t,x}(s;u^*,v)),v(s,x^{t,x}(s;u^*,v),u^*(s,x^{t,x}(s;u^*,v)))\big)\bigg\rangle\\
&\qquad\qquad +\frac{1}{2} \sum\limits_{i,j=1}^n \Big[a_{ij}\big(s,x^{t,x}(s;u^*,v),u^*(s,x^{t,x}(s;u^*,v)),\\
&\qquad\qquad\qquad v(s,x^{t,x}(s;u^*,v),u^*(s,x^{t,x}(s;u^*,v)))\big)\cdot\frac{{\partial}^2 V_2}{\partial x_i \partial x_j}(s,x^{t,x}(s;u^*,v))\Big]\\
&\qquad\qquad+g_2\big(s,x^{t,x}(s;u^*,v),u^*(s,x^{t,x}(s;u^*,v)),v(s,x^{t,x}(s;u^*,v),u^*(s,x^{t,x}(s;u^*,v)))\big)\bigg] ds \Bigg\}\\
&\le J^{t,x}_2(u^*(\cdot,\cdot),v(\cdot,\cdot,u^*(\cdot,\cdot)))-\mathbb E\int_t^T\bigg[\frac{\partial V_2}{\partial s}(s,x^{t,x}(s;u^*,v))\\
&\qquad+\min\limits_{\nu \in\, V}\bigg\{\bigg\langle\frac{\partial V_2}{\partial x}(s,x^{t,x}(s;u^*,v)),f(s,x^{t,x}(s;u^*,v),u^*(s,x^{t,x}(s;u^*,v)),\nu)\bigg\rangle\\
&\qquad\qquad+\frac{1}{2} \sum\limits_{i,j=1}^n \Big[a_{ij}(s,x^{t,x}(s;u^*,v),u^*(s,x^{t,x}(s;u^*,v)),\nu)\cdot\frac{{\partial}^2 V_2}{\partial x_i \partial x_j}(s,x^{t,x}(s;u^*,v))\Big]\\
&\qquad\qquad+g_2(s,x^{t,x}(s;u^*,v),u^*(s,x^{t,x}(s;u^*,v)),\nu)\bigg\}\bigg] ds\\
\end{split}
\end{equation*}
\begin{equation}\label{eq:123}
\hspace{-9cm}
=J^{t,x}_2(u^*(\cdot,\cdot),v(\cdot,\cdot,u^*(\cdot,\cdot))).
\end{equation}
Similarly, if we apply It\^{o}'s formula to $V_2(\cdot,x^{t,x}(\cdot;u^*,v^*))$ and follow the above procedure, then from the definition of $u^*(\cdot,\cdot)$ and $v^*(\cdot,\cdot,\cdot)$, we get
\begin{equation}\label{eq:b}
V_2(t,x)=J^{t,x}_2(u^*(\cdot,\cdot),v^*(\cdot,\cdot,u^*(\cdot,\cdot))).
\end{equation}
From $(\ref{eq:123})$ and $(\ref{eq:b})$, we obtain
\begin{equation}\label{eq:c}
\begin{aligned}
&J^{t,x}_2(u^*(\cdot,\cdot),v^*(\cdot,\cdot,u^*(\cdot,\cdot)))\le J^{t,x}_2(u^*(\cdot,\cdot),v(\cdot,\cdot,u^*(\cdot,\cdot))),\\
&\hspace{4cm} \forall v(\cdot,\cdot,\cdot)\in \mathcal{V}[0,T],\quad \forall (t,x)\in [0,T]\times \mathbb{R}^n.
\end{aligned}
\end{equation}
\noindent Applying It\^{o}'s formula to $V_1(\cdot,x^{t,x}(\cdot;u,v^*))$ and applying analogous limiting argument, we can also obtain
\begin{equation*}
\begin{split}
&V_1(t,x)=\mathbb E\big[V_1(T,x^{t,x}(T;u,v^*))\big]\\
&\quad +\mathbb E\int_t^Tg_1\big(s,x^{t,x}(s;u,v^*),u(s,x^{t,x}(s;u,v^*)),v^*(s,x^{t,x}(s;u,v^*),u(s,x^{t,x}(s;u,v^*)))\big)ds\\
&\quad -\mathbb E\int_t^T\bigg[\frac{\partial V_1}{\partial s}(s,x^{t,x}(s;u,v^*))+\bigg\langle\frac{\partial V_1}{\partial x}(s,x^{t,x}(s;u,v^*)),f\big(s,x^{t,x}(s;u,v^*),\\
&\qquad\qquad\qquad u(s,x^{t,x}(s;u,v^*)),v^*(s,x^{t,x}(s;u,v^*),u(s,x^{t,x}(s;u,v^*)))\big)\bigg\rangle\\
&\qquad\qquad +\frac{1}{2} \sum\limits_{i,j=1}^n \Big[a_{ij}\big(s,x^{t,x}(s;u,v^*),u(s,x^{t,x}(s;u,v^*)),\\
&\qquad\qquad\qquad v^*(s,x^{t,x}(s;u,v^*),u(s,x^{t,x}(s;u,v^*)))\big)\cdot\frac{{\partial}^2 V_1}{\partial x_i \partial x_j}(s,x^{t,x}(s;u,v^*))\Big]\\
&\qquad\qquad +g_1\big(s,x^{t,x}(s;u,v^*),u(s,x^{t,x}(s;u,v^*)),v^*(s,x^{t,x}(s;u,v^*),u(s,x^{t,x}(s;u,v^*)))\big)\bigg]ds\\
&\le J^{t,x}_1(u(\cdot,\cdot),v^*(\cdot,\cdot,u(\cdot,\cdot)))-\mathbb E\int_t^T\bigg[\frac{\partial V_1}{\partial s}(s,x^{t,x}(s;u,v^*))\\
&\qquad +\min\limits_{\mu\in\, U}\bigg\{\bigg\langle\frac{\partial V_1}{\partial x}(s,x^{t,x}(s;u,v^*)),f\big(s,x^{t,x}(s;u,v^*),\mu,v^*(s,x^{t,x}(s;u,v^*),\mu)\big)\bigg\rangle\\
&\qquad\qquad +\frac{1}{2} \sum\limits_{i,j=1}^n \Big[a_{ij}\big(s,x^{t,x}(s;u,v^*),\mu,v^*(s,x^{t,x}(s;u,v^*),\mu)\big)\cdot\frac{{\partial}^2 V_1}{\partial x_i \partial x_j}(s,x^{t,x}(s;u,v^*))\Big]\\
&\qquad\qquad +g_1\big(s,x^{t,x}(s;u,v^*),\mu,v^*(s,x^{t,x}(s;u,v^*),\mu)\big)\bigg\}\bigg] ds\\
&=J^{t,x}_1(u(\cdot,\cdot),v^*(\cdot,\cdot,u(\cdot,\cdot))).\label{eq:d}
\end{split}
\end{equation*}
Analogously, applying It\^{o}'s formula to $V_1(\cdot,x^{t,x}(\cdot;u^*,v^*))$ and proceeding as above, we obtain
\begin{equation}\label{eq:e}
V_1(t,x)=J^{t,x}_1(u^*(\cdot,\cdot),v^*(\cdot,\cdot,u^*(\cdot,\cdot))),
\end{equation}
which implies that
\begin{equation}\label{eq:f}
\begin{aligned}
&J^{t,x}_1(u^*(\cdot,\cdot),v^*(\cdot,\cdot,u^*(\cdot,\cdot)))\le J^{t,x}_1(u(\cdot,\cdot),v^*(\cdot,\cdot,u(\cdot,\cdot))),\\
&\hspace{4cm} \forall u(\cdot,\cdot)\in \mathcal{U}[0,T],\quad \forall (t,x)\in [0,T]\times \mathbb{R}^n.
\end{aligned}
\end{equation}
So we conclude from $(\ref{eq:c})$ and $(\ref{eq:f})$, that $(u^*(\cdot,\cdot),v^*(\cdot,\cdot,u^*(\cdot,\cdot)))$ is a feedback Stackelberg equilibrium. The proof is complete.
\end{proof}

\section{Linear Quadratic Case}\label{Section 4}

In this section, an LQ feedback Stackelberg stochastic differential game is researched. We use Riccati equations to represent the feedback Stackelberg equilibrium. The state equation is
\begin{equation}\label{LQ state SDE-1}
\left\{
\begin{aligned}
dx(s)&=\big[A(s)x(s)+B_1(s)u(s)+B_2(s)v(s)+b(s)\big]ds\\
&\quad+\big[C(s)x(s)+D_1(s)u(s)+D_2(s)v(s)+\lambda(s)\big]dW(s),\quad s \in [t,T],\\
 x(t)&=x,
\end{aligned}
\right.
\end{equation}
where $A(\cdot), B_1(\cdot), B_2(\cdot), C(\cdot), D_1(\cdot), D_2(\cdot), b(\cdot), \lambda(\cdot)$ are given matrix-valued deterministic functions, and $W(\cdot)$ is one-dimensional for notational simplicity. The coefficients of the state equation satisfy the following:
\begin{equation*}\label{LQ state SDE-1 Coeff}
\left\{
\begin{aligned}
&A(\cdot),C(\cdot)\in L^{\infty}(t,T;\mathbb{R}^{n\times n}),\quad B_1(\cdot),D_1(\cdot)\in L^{\infty}(t,T;\mathbb{R}^{n\times {m_1}}),\\
&B_2(\cdot),D_2(\cdot)\in L^{\infty}(t,T;\mathbb{R}^{n\times {m_2}}),\quad b(\cdot),\lambda(\cdot)\in L^2(t,T;\mathbb{R}^n).
\end{aligned}
\right.
\end{equation*}
Next, for $i=1,2,$ we introduce the following cost functionals:
\begin{equation}\label{LQ cost functionals-leader and follower}
\begin{aligned}
&J_i(t,x;u(\cdot),v(\cdot))=\frac{1}{2}\mathbb E\Bigg\{\mathlarger{\int^T_t} \left[\left\langle\left( \begin{array}{ccc}
Q_i(s) & M_{i1}(s)^\top & M_{i2}(s)^\top\\
M_{i1}(s) & R_{i11}(s) & R_{i12}(s)\\
M_{i2}(s) & R_{i21}(s) & R_{i22}(s)\\
\end{array} \right)
\left( \begin{array}{ccc}
x(s)\\
u(s)\\
v(s)\\
\end{array} \right),
\left( \begin{array}{ccc}
x(s)\\
u(s)\\
v(s)\\
\end{array} \right)\right\rangle\right.\\
&\qquad\qquad+2\left.\left\langle \left( \begin{array}{ccc}
q_i(s)\\
\rho_{i1}(s)\\
\rho_{i2}(s)\\
\end{array} \right),
\left( \begin{array}{ccc}
x(s)\\
u(s)\\
v(s)\\
\end{array} \right)
\right\rangle
\right]ds+\langle L_i x(T), x(T)\rangle+2\langle N_i, x(T)\rangle\Bigg\},
\end{aligned}
\end{equation}
where $Q_i(\cdot), M_{i1}(\cdot), M_{i2}(\cdot), R_{i11}(\cdot), R_{i12}(\cdot), R_{i21}(\cdot), R_{i22}(\cdot), q_i(\cdot), \rho_{i1}(\cdot), \rho_{i2}(\cdot)$ are given matrix\\-valued deterministic functions, $L_i$ is a $n \times n$ symmetric matrix and $N_i \in \mathbb{R}^n$. The weighting functions in the cost functionals satisfy the following:
\begin{equation*}\label{LQ cost functionals Coeff}
\left\{
\begin{aligned}
&Q_i(\cdot)\in L^\infty(t,T;\mathcal{S}^n),\quad M_{i1}(\cdot)\in L^{\infty}(t,T;\mathbb{R}^{{m_1} \times n}), \qquad M_{i2}(\cdot)\in L^{\infty}(t,T;\mathbb{R}^{{m_2} \times n}),\\
&R_{i11}(\cdot)\in L^\infty(t,T;\mathcal{S}^{m_1}), \quad R_{i12}(\cdot)={R_{i21}(\cdot)}^\top\in L^\infty(t,T;\mathbb{R}^{{m_1}\times{m_2}}), \quad R_{i22}(\cdot)\in L^\infty(t,T;\mathcal{S}^{m_2}),\\
&q_i(\cdot)\in L^\infty(t,T;\mathbb{R}^n),\quad {\rho}_{i1}(\cdot)\in L^\infty(t,T;\mathbb{R}^{m_1}),\quad {\rho}_{i2}(\cdot)\in L^\infty(t,T;\mathbb{R}^{m_2}).
\end{aligned}
\right.
\end{equation*}
First of all, for the leader's every action $\mu \in U,$ we compute the follower's unique optimal response function:
\begin{equation*}
\begin{aligned}
&T_2(s,x,\mu,p_2,A''):=\mathop{argmin}\limits_{\nu \in\, V} H_2(s,x,\mu,\nu,p_2,A'')\\
&:=\mathop{argmin}\limits_{\nu \in\, V}\Big\{\langle p_2, A(s)x+B_1(s) \mu +B_2(s) \nu +b(s)\rangle+\frac{1}{2} \mathrm{tr} \big[\sigma(s,x,\mu,\nu)\sigma(s,x,\mu,\nu)^\top A''\big]\\
&\qquad\qquad\qquad+g_2(s,x,\mu,\nu)\Big\}\\
&:=\mathop{argmin}\limits_{\nu \in\, V}\Big\{\langle p_2, A(s)x+B_1(s) \mu +B_2(s) \nu +b(s)\rangle+\frac{1}{2} \langle A'' \sigma(s,x,\mu,\nu), \sigma(s,x,\mu,\nu) \rangle\\
&\qquad\qquad\qquad+g_2(s,x,\mu,\nu)\Big\},\quad s \in [t,T].
\end{aligned}
\end{equation*}
Using completion of squares, we get
\begin{equation*}
\begin{aligned}
&H_2(s,x,\mu,\nu,p_2,A'')\\
&=\frac{1}{2}\big|\hat{R}_2 (s)^{\frac{1}{2}}[\nu+\Psi (s)x+\Phi(s)\mu+\psi(s)]\big|^2\\
&\quad+\frac{1}{2}\langle(C(s)^\top A'' C(s)+Q_2(s)-\Psi(s)^\top\hat{R}_2 (s)\Psi(s))x,x\rangle\\
&\quad+\langle A(s)^\top p_2 +C(s)^\top A''D_1(s)\mu+C(s)^\top A''\lambda(s)+M_{21}(s)^\top \mu+q_2(s)-\Psi(s)^\top\hat{R}_2 (s)\Phi(s)\mu\\
&\qquad-\Psi(s)^\top\hat{R}_2(s)\psi(s),x\rangle+\frac{1}{2}\langle(D_1(s)^\top A''D_1(s)+R_{211}(s)-\Phi(s)^\top \hat{R}_2 (s)\Phi(s))\mu,\mu\rangle\\
&\quad+\langle B_1(s)^\top p_2+D_1(s)^\top A'' \lambda(s)+\rho_{21}(s)-\Phi(s)^\top\hat{R}_2 (s)\psi(s),\mu\rangle\\
&\quad+{p_2}^\top b(s)+\frac{1}{2} \lambda(s)^\top A'' \lambda(s)-\frac{1}{2} \psi(s)^\top\hat{R}_2 (s)\psi(s),\quad s \in [t,T],
\end{aligned}
\end{equation*}
where
\begin{equation*}\label{LQ state SDE-1 Coeff--}
\left\{
\begin{aligned}
&\hat{R}_2(s):= D_2(s)^\top A''D_2(s)+R_{222}(s),\\
&\Psi(s):= \hat{R}_2(s)^{-1}(D_2(s)^\top A''C(s)+M_{22}(s)),\\
&\Phi(s):= \hat{R}_2(s)^{-1}(D_2(s)^\top A''D_1(s)+R_{212}(s)^\top),\\
&\psi(s):= \hat{R}_2(s)^{-1}(B_2(s)^\top p_2+D_2(s)^\top A''\lambda(s)+\rho_{22}(s)),\quad s \in [t,T].
\end{aligned}
\right.
\end{equation*}
Thus, by assuming $\hat{R}_2(s)>0, s \in [t,T],$ we obtain from the above that
\begin{equation}\label{g}
\begin{aligned}
T_2(s,x,\mu,p_2,A'')=-\Psi(s)x-\Phi(s)\mu-\psi(s),\quad s \in [t,T].
\end{aligned}
\end{equation}
Next, the leader's optimal action is:
\begin{equation*}
\begin{aligned}
&T_1(s,x,p_1,p_2,A',A''):=\mathop{argmin}\limits_{\mu \in\, U} H_1(s,x,\mu,T_2(s,x,\mu,p_2,A''),p_1,A')\\
&:=\mathop{argmin}\limits_{\mu \in\, U}\Big\{\langle p_1, A(s)x+B_1(s) \mu +B_2(s)(-\Psi(s)x-\Phi(s)\mu-\psi(s)) +b(s)\rangle\\
&\qquad\qquad\qquad+\frac{1}{2} \mathrm{tr} \big[\sigma(s,x,\mu,-\Psi(s)x-\Phi(s)\mu-\psi(s))\sigma(s,x,\mu,-\Psi(s)x-\Phi(s)\mu-\psi(s))^\top A'\big]\\
&\qquad\qquad\qquad+g_1(s,x,\mu,-\Psi(s)x-\Phi(s)\mu-\psi(s))\Big\}\\
&:=\mathop{argmin}\limits_{\mu \in\, U}\Big\{\langle p_1, A(s)x+B_1(s) \mu +B_2(s)(-\Psi(s)x-\Phi(s)\mu-\psi(s)) +b(s)\rangle\\
&\qquad\qquad\qquad+\frac{1}{2} \langle A' \sigma(s,x,\mu,-\Psi(s)x-\Phi(s)\mu-\psi(s)),\sigma(s,x,\mu,-\Psi(s)x-\Phi(s)\mu-\psi(s)) \rangle\\
&\qquad\qquad\qquad+g_1(s,x,\mu,-\Psi(s)x-\Phi(s)\mu-\psi(s))\Big\},\quad s \in [t,T].\\
\end{aligned}
\end{equation*}
In order to get $T_1(s,x,p_1,p_2,A',A'')$, we calculate
\begin{equation*}
\begin{aligned}
&H_1(s,x,\mu,T_2(s,x,\mu,p_2,A''),p_1,A')\\
&={\mu}^\top \Big[\frac{1}{2}D_1(s)^\top A'D_1(s)-{\Phi(s)}^\top D_2(s)^\top A'D_1(s)+\frac{1}{2}{\Phi(s)}^\top D_2(s)^\top A'D_2(s)\Phi(s)+\\
&\qquad\quad\frac{1}{2}R_{111}(s)-{\Phi(s)}^\top R_{112}(s)^\top+\frac{1}{2}{\Phi(s)}^\top R_{122}(s)\Phi(s)\Big]\mu\\
&\quad+\Big[p_1^\top B_1(s)-p_1^\top B_2(s)\Phi(s)+x^\top C(s)^\top A'D_1(s)-x^\top C(s)^\top A'D_2(s)\Phi(s)\\
&\qquad-x^\top \Psi(s)^\top D_2(s)^\top A'D_1(s)-\psi(s)^\top D_2(s)^\top A'D_1(s)+\lambda(s)^\top A'D_1(s)\\
&\qquad+x^\top \Psi(s)^\top D_2(s)^\top A'D_2(s)\Phi(s)+\psi(s)^\top D_2(s)^\top A'D_2(s)\Phi(s)-\lambda(s)^\top A'D_2(s)\Phi(s)\\
&\qquad+x^\top M_{11}(s)^\top-x^\top M_{12}(s)^\top\Phi(s)-x^\top \Psi(s)^\top R_{112}(s)^\top-\psi(s)^\top R_{112}(s)^\top\\
&\qquad+x^\top \Psi(s)^\top R_{122}(s)\Phi(s)+\psi(s)^\top R_{122}(s)\Phi(s)+\rho_{11}(s)^\top-\rho_{12}(s)^\top \Phi(s)\Big]\mu\\
&\quad+x^\top \Big[\frac{1}{2}C(s)^\top A'C(s)-C(s)^\top A'D_2(s)\Psi(s)+\frac{1}{2}{\Psi(s)}^\top D_2(s)^\top A'D_2(s)\Psi(s)+\frac{1}{2}Q_1(s)\\
&\qquad-\Psi(s)^\top M_{12}(s)+\frac{1}{2}{\Psi(s)}^\top R_{122}(s)\Psi(s)\Big]x\\
&\quad+\Big[p_1^\top A(s)-p_1^\top B_2(s)\Psi(s)-\psi(s)^\top D_2(s)^\top A'C(s)+\lambda(s)^\top A'C(s)\\
&\qquad+{\psi(s)}^\top D_2(s)^\top A'D_2(s)\Psi(s)-\lambda(s)^\top A'D_2(s)\Psi(s)-\psi(s)^\top M_{12}(s)\\
&\qquad+{\psi(s)}^\top R_{122}(s)\Psi(s)+q_1(s)^\top-\rho_{12}(s)^\top \Psi(s)\Big]x\\
&\quad-p_1^\top B_2(s)\psi(s)+p_1^\top b(s)+\frac{1}{2}{\psi(s)}^\top D_2(s)^\top A'D_2(s)\psi(s)-{\psi(s)}^\top D_2(s)^\top A'\lambda(s)\\
&\quad+\frac{1}{2}\lambda(s)^\top A'\lambda(s)+\frac{1}{2}\psi(s)^\top R_{122}(s)\psi(s)-{\rho}_{12}(s)^\top \psi(s),\quad s \in [t,T].
\end{aligned}
\end{equation*}
Then
\begin{equation*}
\frac{\partial H_1(s,x,\mu,T_2(s,x,\mu,p_2,A''),p_1,A')}{\partial \mu}=\frac{\hat{R}_1(s)+\hat{R}_1(s)^\top}{2}\mu+Y(s)^\top,\quad s \in [t,T],
\end{equation*}
\begin{equation*}
\begin{aligned}
\frac{{\partial}^2 H_1(s,x,\mu,T_2(s,x,\mu,p_2,A''),p_1,A')}{{\partial \mu}^2}=\frac{\hat{R}_1(s)+\hat{R}_1(s)^\top}{2},\quad s \in [t,T],
\end{aligned}
\end{equation*}
where
\begin{equation*}
\begin{aligned}
\hat{R}_1(s)&:= D_1(s)^\top A'D_1(s)-2\Phi(s)^\top D_2(s)^\top A'D_1(s)+\Phi(s)^\top D_2(s)^\top A'D_2(s)\Phi(s)\\
            &\quad +R_{111}(s)-2\Phi(s)^\top R_{112}(s)^\top+{\Phi(s)}^\top R_{122}(s)\Phi(s),\\
        Y(s)&:= p_1^\top B_1(s)-p_1^\top B_2(s)\Phi(s)+x^\top C(s)^\top A'D_1(s)-x^\top C(s)^\top A'D_2(s)\Phi(s)\\
            &\quad -x^\top \Psi(s)^\top D_2(s)^\top A'D_1(s)-\psi(s)^\top D_2(s)^\top A'D_1(s)+\lambda(s)^\top A'D_1(s)\\
            &\quad +x^\top \Psi(s)^\top D_2(s)^\top A'D_2(s)\Phi(s)+\psi(s)^\top D_2(s)^\top A'D_2(s)\Phi(s)-\lambda(s)^\top A'D_2(s)\Phi(s)\\
            &\quad +x^\top M_{11}(s)^\top-x^\top M_{12}(s)^\top\Phi(s)-x^\top \Psi(s)^\top R_{112}(s)^\top-\psi(s)^\top R_{112}(s)^\top\\
            &\quad +x^\top \Psi(s)^\top R_{122}(s)\Phi(s)+\psi(s)^\top R_{122}(s)\Phi(s)+\rho_{11}(s)^\top-\rho_{12}(s)^\top \Phi(s).
\end{aligned}
\end{equation*}
Thus, by assuming
\begin{equation*}
\left\{
\begin{aligned}
&\frac{\hat{R}_1(s)+\hat{R}_1(s)^\top}{2}\mu_0+Y(s)^\top=0,\qquad \frac{\hat{R}_1(s)+\hat{R}_1(s)^\top}{2}>0,\\
&\mu_0 \in \mathcal{U}[0,T],\quad s \in [t,T],
\end{aligned}
\right.
\end{equation*}
we get
\begin{equation}\label{h}
\begin{aligned}
T_1(s,x,p_1,p_2,A',A'')=\mu_0=-2(\hat{R}_1(s)+\hat{R}_1(s)^\top)^{-1}Y(s)^\top,\quad s \in [t,T].
\end{aligned}
\end{equation}
Substituting $(\ref{g})$ and $(\ref{h})$ into $(\ref{PDE-1})$ and $(\ref{PDE-2})$, when
\[p_1=\frac{\partial V_1}{\partial x},\quad p_2=\frac{\partial V_2}{\partial x},\quad A'=\frac{{\partial}^2 V_1}{{\partial x}^2},\quad A''=\frac{{\partial}^2 V_2}{{\partial x}^2},\]
we define
\begin{equation*}
\begin{aligned}
\bar{\hat{R}}_2(s)& := \hat{R}_2(s){\bigg\vert}_{p_1=\frac{\partial V_1}{\partial x},\ p_2=\frac{\partial V_2}{\partial x},\ A'=\frac{{\partial}^2 V_1}{{\partial x}^2},\ A''=\frac{{\partial}^2 V_2}{{\partial x}^2}}\\
&=D_2(s)^\top \frac{{\partial}^2 V_2}{{\partial x}^2} D_2(s)+R_{222}(s),\\
\bar{\Psi}(s)& := {\Psi (s)}{\bigg\vert}_{p_1=\frac{\partial V_1}{\partial x},\ p_2=\frac{\partial V_2}{\partial x},\ A'=\frac{{\partial}^2 V_1}{{\partial x}^2},\ A''=\frac{{\partial}^2 V_2}{{\partial x}^2}}\\
&=\bar{\hat{R}}_2(s)^{-1}\left(D_2(s)^\top \frac{{\partial}^2 V_2}{{\partial x}^2}C(s)+M_{22}(s)\right),\\
\bar{\Phi}(s)& := {\Phi(s)}{\bigg\vert}_{p_1=\frac{\partial V_1}{\partial x},\ p_2=\frac{\partial V_2}{\partial x},\ A'=\frac{{\partial}^2 V_1}{{\partial x}^2},\ A''=\frac{{\partial}^2 V_2}{{\partial x}^2}}\\
&=\bar{\hat{R}}_2(s)^{-1}\left(D_2(s)^\top \frac{{\partial}^2 V_2}{{\partial x}^2}D_1(s)+R_{212}(s)^\top\right),\\
\bar{\psi}(s)& := {\psi(s)}{\bigg\vert}_{p_1=\frac{\partial V_1}{\partial x},\ p_2=\frac{\partial V_2}{\partial x},\ A'=\frac{{\partial}^2 V_1}{{\partial x}^2},\ A''=\frac{{\partial}^2 V_2}{{\partial x}^2}}\\
&=\bar{\hat{R}}_2(s)^{-1}\left(B_2(s)^\top \frac{{\partial} V_2}{\partial x}+D_2(s)^\top \frac{{\partial}^2 V_2}{{\partial x}^2}{\lambda}(s)+{\rho}_{22}(s)\right),
\end{aligned}
\end{equation*}
\begin{equation*}
\begin{aligned}
\bar{\mu}_0& := {\mu}_0{\bigg\vert}_{p_1=\frac{\partial V_1}{\partial x},\ p_2=\frac{\partial V_2}{\partial x},\ A'=\frac{{\partial}^2 V_1}{{\partial x}^2},\ A''=\frac{{\partial}^2 V_2}{{\partial x}^2}}\\
&=\bigg[\bigg(D_1(s)^\top \frac{{\partial}^2 V_1}{{\partial x}^2}-\bar{\Phi}(s)^\top D_2(s)^\top \frac{{\partial}^2 V_1}{{\partial x}^2}\bigg)(D_2(s)\bar{\Phi}(s)-D_1(s))\\
&\qquad+\bar{\Phi}(s)^\top\big(R_{112}(s)^\top-R_{122}(s)\bar{\Phi}(s)\big)-R_{111}(s)+R_{112}(s)\bar{\Phi}(s)\bigg]^{-1}\\
&\quad\times\bigg[\frac{\partial V_1}{\partial x}^\top\big(B_1(s)-B_2(s)\bar{\Phi}(s)\big)+\bigg(-x^\top C(s)^\top \frac{{\partial}^2 V_1}{{\partial x}^2}+x^\top {\bar{\Psi}(s)}^\top D_2(s)^\top \frac{{\partial}^2 V_1}{{\partial x}^2}\\
&\qquad +\bar{\psi}(s)^\top D_2(s)^\top\frac{{\partial}^2 V_1}{{\partial x}^2}-{\lambda(s)}^\top \frac{{\partial}^2 V_1}{{\partial x}^2}\bigg)\big(D_2(s)\bar{\Phi}(s)-D_1(s)\big)\\
&\qquad +x^\top M_{11}(s)^\top-x^\top M_{12}(s)^\top \bar{\Phi}(s)+\big(x^\top{\bar{\Psi}(s)}^\top+\bar{\psi}(s)^\top\big)\\
&\qquad \times\big(R_{122}(s)\bar{\Phi}(s)-R_{112}(s)^\top\big)+\rho_{11}(s)^\top-\rho_{12}(s)^\top\bar{\Phi}(s)\bigg]^\top,\quad s \in [t,T].
\end{aligned}
\end{equation*}
Then we get the resulting PDEs system in the LQ problem:
\begin{equation*}
\begin{aligned}
&\frac{\partial V_1}{\partial s}(s,x)+{\bar{\mu}_0}^\top\bigg[\frac{1}{2}D_1(s)^\top \frac{{\partial}^2 V_1}{{\partial x}^2} D_1(s)-\bar{\Phi}(s)^\top D_2(s)^\top \frac{{\partial}^2 V_1}{{\partial x}^2} D_1(s)\\
&\quad +\frac{1}{2}\bar{\Phi}(s)^\top D_2(s)^\top \frac{{\partial}^2 V_1}{{\partial x}^2} D_2(s)\bar{\Phi}(s)+\frac{1}{2}R_{111}(s)-\bar{\Phi}(s)^\top R_{112}(s)^\top+\frac{1}{2}\bar{\Phi}(s)^\top R_{122}(s)\bar{\Phi}(s)\bigg]\bar{\mu}_0\\
&+\bigg[\frac{\partial V_1}{\partial x}^\top\big(B_1(s)-B_2(s)\bar{\Phi}(s)\big)+\bigg(-x^\top C(s)^\top \frac{{\partial}^2 V_1}{{\partial x}^2}+x^\top {\bar{\Psi}(s)}^\top D_2(s)^\top \frac{{\partial}^2 V_1}{{\partial x}^2}\\
&\qquad+\bar{\psi}(s)^\top D_2(s)^\top\frac{{\partial}^2 V_1}{{\partial x}^2}-{\lambda(s)}^\top \frac{{\partial}^2 V_1}{{\partial x}^2}\bigg)(D_2(s)\bar{\Phi}(s)-D_1(s))+x^\top M_{11}(s)^\top\\
&\qquad-x^\top M_{12}(s)^\top \bar{\Phi}(s)+(x^\top{\bar{\Psi}(s)}^\top+\bar{\psi}(s)^\top)(R_{122}(s)\bar{\Phi}(s)-R_{112}(s)^\top)\\
&\qquad+\rho_{11}(s)^\top-\rho_{12}(s)^\top\bar{\Phi}(s)\bigg]\bar{\mu}_0\\
&+x^\top\bigg(\frac{1}{2}C(s)^\top \frac{{\partial}^2 V_1}{{\partial x}^2} C(s)-C(s)^\top \frac{{\partial}^2 V_1}{{\partial x}^2} D_2(s)\bar{\Psi}(s)+\frac{1}{2}\bar{\Psi}(s)^\top D_2(s)^\top \frac{{\partial}^2 V_1}{{\partial x}^2} D_2(s)\bar{\Psi}(s)\\
&\quad\qquad+\frac{1}{2} Q_1(s)-\bar{\Psi}(s)^\top M_{12}(s)+\frac{1}{2}\bar{\Psi}(s)^\top R_{122}(s) \bar{\Psi}(s)\bigg)x\\
&+\bigg[\frac{\partial V_1}{\partial x}^\top \big(A(s)-B_2(s)\bar{\Psi}(s)\big)+\bar{\psi}(s)^\top D_2(s)^\top \frac{{\partial}^2 V_1}{{\partial x}^2}\big(D_2(s)\bar{\Psi}(s)-C(s)\big)\\
&\qquad+{\lambda}(s)^\top \frac{{\partial}^2 V_1}{{\partial x}^2}\big(C(s)-D_2(s)\bar{\Psi}(s)\big)+\bar{\psi}(s)^\top\big(R_{122}(s)\bar{\Psi}(s)-M_{12}(s)\big)\\
&\qquad+q_1(s)^\top-\rho_{12}(s)^\top\bar{\Psi}(s)\bigg]x
\end{aligned}
\end{equation*}
\begin{equation}\label{PDE-3}
\begin{aligned}
&+\frac{\partial V_1}{\partial x}^\top\big(b(s)-B_2(s)\bar{\psi}(s)\big)+\frac{1}{2}\bar{\psi}(s)^\top D_2(s)^\top \frac{{\partial}^2 V_1}{{\partial x}^2}D_2(s)\bar{\psi}(s)-\bar{\psi}(s)^\top D_2(s)^\top \frac{{\partial}^2 V_1}{{\partial x}^2}\lambda(s)\\
&+\frac{1}{2}\lambda(s)^\top \frac{{\partial}^2 V_1}{{\partial x}^2}\lambda(s)+\frac{1}{2}\bar{\psi}(s)^\top R_{122}(s)\bar{\psi}(s)-\rho_{12}(s)^\top \bar{\psi}(s)=0,\quad s \in [t,T],
\end{aligned}
\end{equation}
with $V_1(T,x)=\frac{1}{2}\langle L_1 x,x\rangle+\langle N_1,x\rangle$;
\begin{equation}\label{PDE-4}
\begin{aligned}
&\frac{\partial V_2}{\partial s}(s,x)+\frac{1}{2}x^\top\bigg[C(s)^\top \frac{{\partial}^2 V_2}{{\partial x}^2} C(s)+Q_2(s)-\bar{\Psi}(s)^\top \bar{\hat{R}}_2(s)\bar{\Psi}(s)\bigg]x\\
&+\bigg[\frac{\partial V_2}{\partial x}^\top A(s)+{\bar{\mu}_0}^\top D_1(s)^\top \frac{{\partial}^2 V_2}{{\partial x}^2} C(s)+\lambda(s)^\top \frac{{\partial}^2 V_2}{{\partial x}^2}C(s)\\
&\qquad +{\bar{\mu}_0}^\top M_{21}(s)+q_2(s)^\top-{\bar{\mu}_0}^\top \bar{\Phi}(s)^\top \bar{\hat{R}}_2(s)\bar{\Psi}(s)-\bar{\psi}(s)^\top \bar{\hat{R}}_2(s)\bar{\Psi}(s)\bigg]x\\
&+\frac{1}{2}{\bar{\mu}_0}^\top\bigg[D_1(s)^\top \frac{{\partial}^2 V_2}{{\partial x}^2} D_1(s)+R_{211}(s)-\bar{\Phi}(s)^\top \bar{\hat{R}}_2(s)\bar{\Phi}(s) \bigg]{\bar{\mu}_0}\\
&+\bigg[\frac{\partial V_2}{\partial x}^\top B_1(s)+\lambda(s)^\top \frac{{\partial}^2 V_2}{{\partial x}^2}D_1(s)+\rho_{21}(s)^\top-\bar{\psi}(s)^\top \bar{\hat{R}}_2(s)\bar{\Phi}(s)\bigg]\bar{\mu}_0\\
&+\frac{\partial V_2}{\partial x}^\top b(s)+\frac{1}{2}\lambda(s)^\top \frac{{\partial}^2 V_2}{{\partial x}^2}\lambda(s)-\frac{1}{2}\bar{\psi}(s)^\top \bar{\hat{R}}_2(s)\bar{\psi}(s)=0,\quad s \in [t,T],
\end{aligned}
\end{equation}
with $V_2(T,x)=\frac{1}{2}\langle L_2 x,x\rangle+\langle N_2,x\rangle$.

In general, since the two PDEs in the above system are coupled and have complex structure, it is very difficult to solve this system explicitly to obtain its solution $V_1$ and $V_2$. In the following two subsections, we consider some special cases.

\subsection{Case 1}\label{Case 1}
Let $b(\cdot)=0,D_1(\cdot)=0,D_2(\cdot)=0,\lambda(\cdot)=0,M_{11}(\cdot)=0,M_{12}(\cdot)=0,M_{21}(\cdot)=0,M_{22}(\cdot)=0,R_{122}(\cdot)=0,R_{211}(\cdot)=0,q_1(\cdot)=0,
\rho_{11}(\cdot)=0,\rho_{12}(\cdot)=0,q_2(\cdot)=0,\rho_{21}(\cdot)=0,\rho_{22}(\cdot)=0,N_1=0,N_2=0$ in $(\ref{LQ state SDE-1})$ and $(\ref{LQ cost functionals-leader and follower})$. Therefore, the state equation and cost functionals are of the following form:
\begin{equation}\label{i}
\left\{
\begin{aligned}
dx(s)&=(A(s)x(s)+B_1(s)u(s)+B_2(s)v(s))ds+C(s)x(s)dW(s),\quad s \in [t,T],\\
 x(t)&=x,
\end{aligned}
\right.
\end{equation}
and
\begin{equation}\label{j}
\begin{aligned}
J_1(t,x;u(\cdot),v(\cdot))=\frac{1}{2}\mathbb{E}&\bigg\{\int_t^T\big[x(s)^\top Q_1(s)x(s)+u(s)^\top R_{111}(s)u(s)+2u(s)^\top R_{112}(s)v(s)\big]ds\\
&\quad+\langle L_1 x(T),x(T)\rangle\bigg\},\\
J_2(t,x;u(\cdot),v(\cdot))=\frac{1}{2}\mathbb{E}&\bigg\{\int_t^T\big[x(s)^\top Q_2(s)x(s)+v(s)^\top R_{222}(s)v(s)+2u(s)^\top R_{212}(s)v(s)\big]ds\\
&\quad+\langle L_2 x(T),x(T)\rangle\bigg\}.
\end{aligned}
\end{equation}
Moreover, we can get
\begin{equation*}
\begin{aligned}
&\bar{\hat{R}}_2(s)=R_{222}(s),\qquad\qquad\quad \bar{\Psi}(s)=0,\\
&\bar{\Phi}(s)=R_{222}(s)^{-1}R_{212}(s)^\top, \quad\bar{\psi}(s)=R_{222}(s)^{-1}B_2(s)^\top\frac{\partial V_2}{\partial x},\quad s \in [t,T].
\end{aligned}
\end{equation*}
By assuming
\begin{equation*}
\left\{
\begin{aligned}
&R_{222}(s)>0,\\
&R_{111}(s)-R_{212}(s)R_{222}(s)^{-1}R_{112}(s)^\top-R_{112}(s)R_{222}(s)^{-1}R_{212}(s)^\top>0,\quad s \in [t,T],
\end{aligned}
\right.
\end{equation*}
we obtain
\begin{equation*}
\begin{aligned}
\bar{\mu}_0=&\bigg[R_{212}(s)R_{222}(s)^{-1}R_{112}(s)^\top+R_{112}(s)R_{222}(s)^{-1}R_{212}(s)^\top-R_{111}(s)\bigg]^{-1}\\
&\times\bigg[(B_1(s)^\top-R_{212}(s)R_{222}(s)^{-1}B_2(s)^\top)\frac{\partial V_1}{\partial x}-R_{112}(s)R_{222}(s)^{-1}B_2(s)^\top \frac{\partial V_2}{\partial x}\bigg],\quad s \in [t,T].
\end{aligned}
\end{equation*}

In this case, $(\ref{PDE-3})$ becomes
\begin{equation}\label{PDE-5}
\begin{aligned}
&\frac{\partial V_1}{\partial s}(s,x)+\bigg[(B_1(s)^\top-R_{212}(s)R_{222}(s)^{-1}B_2(s)^\top)\frac{\partial V_1}{\partial x}-R_{112}(s)R_{222}(s)^{-1}B_2(s)^\top \frac{\partial V_2}{\partial x}\bigg]^\top\\
&\quad\times\bigg[R_{212}(s)R_{222}(s)^{-1}R_{112}(s)^\top+R_{112}(s)R_{222}(s)^{-1}R_{212}(s)^\top-R_{111}(s)\bigg]^{-1}\\
&\quad\times\bigg[\frac{1}{2}R_{111}(s)-R_{212}(s)R_{222}(s)^{-1}R_{112}(s)^\top\bigg]\\
&\quad\times\bigg[R_{212}(s)R_{222}(s)^{-1}R_{112}(s)^\top+R_{112}(s)R_{222}(s)^{-1}R_{212}(s)^\top-R_{111}(s)\bigg]^{-1}\\
&\quad\times\bigg[(B_1(s)^\top-R_{212}(s)R_{222}(s)^{-1}B_2(s)^\top)\frac{\partial V_1}{\partial x}-R_{112}(s)R_{222}(s)^{-1}B_2(s)^\top \frac{\partial V_2}{\partial x}\bigg]\\
&+\bigg[\big(B_1(s)^\top-R_{212}(s)R_{222}(s)^{-1}B_2(s)^\top\big)\frac{\partial V_1}{\partial x}-R_{112}(s)R_{222}(s)^{-1}B_2(s)^\top \frac{\partial V_2}{\partial x}\bigg]^\top\\
&\quad\times\bigg[R_{212}(s)R_{222}(s)^{-1}R_{112}(s)^\top+R_{112}(s)R_{222}(s)^{-1}R_{212}(s)^\top-R_{111}(s)\bigg]^{-1}\\
&\quad\times\bigg[\big(B_1(s)^\top-R_{212}(s)R_{222}(s)^{-1}B_2(s)^\top\big)\frac{\partial V_1}{\partial x}-R_{112}(s)R_{222}(s)^{-1}B_2(s)^\top \frac{\partial V_2}{\partial x}\bigg]\\
&+\frac{1}{2}x^\top\bigg[C(s)^\top\frac{{\partial}^2 V_1}{{\partial x}^2}C(s)+Q_1(s)\bigg]x\\
&+\frac{\partial V_1}{\partial x}^\top A(s)x-\frac{\partial V_1}{\partial x}^\top B_2(s)R_{222}(s)^{-1}B_2(s)^\top\frac{\partial V_2}{\partial x}=0,\quad s \in [t,T],
\end{aligned}
\end{equation}
with $V_1(T,x)=\frac{1}{2}\langle L_1 x,x\rangle$; and $(\ref{PDE-4})$ becomes
\begin{equation}\label{PDE-6}
\begin{aligned}
&\frac{\partial V_2}{\partial s}(s,x)-\frac{1}{2}\bigg[(B_1(s)^\top-R_{212}(s)R_{222}(s)^{-1}B_2(s)^\top)\frac{\partial V_1}{\partial x}-R_{112}(s)R_{222}(s)^{-1}B_2(s)^\top \frac{\partial V_2}{\partial x}\bigg]^\top\\
&\qquad\times\bigg[R_{212}(s)R_{222}(s)^{-1}R_{112}(s)^\top+R_{112}(s)R_{222}(s)^{-1}R_{212}(s)^\top-R_{111}(s)\bigg]^{-1}\\
&\qquad\times R_{212}(s)R_{222}(s)^{-1}R_{212}(s)^\top\\
&\qquad\times\bigg[R_{212}(s)R_{222}(s)^{-1}R_{112}(s)^\top+R_{112}(s)R_{222}(s)^{-1}R_{212}(s)^\top-R_{111}(s)\bigg]^{-1}\\
&\qquad\times\bigg[(B_1(s)^\top-R_{212}(s)R_{222}(s)^{-1}B_2(s)^\top)\frac{\partial V_1}{\partial x}-R_{112}(s)R_{222}(s)^{-1}B_2(s)^\top \frac{\partial V_2}{\partial x}\bigg]\\
&+\bigg[\big(B_1(s)^\top-R_{212}(s)R_{222}(s)^{-1}B_2(s)^\top\big)\frac{\partial V_2}{\partial x}\bigg]^\top\\
&\quad\times\bigg[R_{212}(s)R_{222}(s)^{-1}R_{112}(s)^\top+R_{112}(s)R_{222}(s)^{-1}R_{212}(s)^\top-R_{111}(s)\bigg]^{-1}\\
&\quad\times\bigg[\big(B_1(s)^\top-R_{212}(s)R_{222}(s)^{-1}B_2(s)^\top\big)\frac{\partial V_1}{\partial x}-R_{112}(s)R_{222}(s)^{-1}B_2(s)^\top \frac{\partial V_2}{\partial x}\bigg]\\
&+\frac{1}{2}x^\top\bigg[C(s)^\top\frac{{\partial}^2 V_2}{{\partial x}^2}C(s)+Q_2(s)\bigg]x\\
&+\frac{\partial V_2}{\partial x}^\top A(s)x-\frac{1}{2}\frac{\partial V_2}{\partial x}^\top B_2(s)R_{222}(s)^{-1}B_2(s)^\top\frac{\partial V_2}{\partial x}=0,\quad s \in [t,T],
\end{aligned}
\end{equation}
with $V_2(T,x)=\frac{1}{2}\langle L_2 x,x\rangle$.

We conjecture solutions of the following quadratic form:
\begin{equation}\label{quadratic Value}
\begin{aligned}
V_1(s,x)&=\frac{1}{2}\langle P_1(s) x,x\rangle,\\
V_2(s,x)&=\frac{1}{2}\langle P_2(s) x,x\rangle,
\end{aligned}
\end{equation}
for some suitable $P_1(\cdot)$ and $P_2(\cdot)$ (where $P_1(\cdot)$ and $P_2(\cdot)$ are $n\times n$ symmetric matrices) with
\[P_1(T)=L_1,\qquad P_2(T)=L_2.\]
Substituting $(\ref{quadratic Value})$ into $(\ref{PDE-5})$ and $(\ref{PDE-6})$, and comparing the quadratic terms in $x$, we get the following system of Riccati equations:
\begin{equation}\label{Riccati eq-1}
\begin{aligned}
&\dot{P}_1(s)+\bigg[\big(B_1(s)^\top-R_{212}(s)R_{222}(s)^{-1}B_2(s)^\top\big)P_1(s)-R_{112}(s)R_{222}(s)^{-1}B_2(s)^\top P_2(s)\bigg]^\top\\
&\quad\times\bigg[R_{212}(s)R_{222}(s)^{-1}R_{112}(s)^\top+R_{112}(s)R_{222}(s)^{-1}R_{212}(s)^\top-R_{111}(s)\bigg]^{-1}\\
&\quad\times\bigg[\big(B_1(s)^\top-R_{212}(s)R_{222}(s)^{-1}B_2(s)^\top\big)P_1(s)-R_{112}(s)R_{222}(s)^{-1}B_2(s)^\top P_2(s)\bigg]\\
&+C(s)^\top P_1(s)C(s)+Q_1(s)+P_1(s)A(s)+A(s)^\top P_1(s)\\
&-P_1(s)B_2(s)R_{222}(s)^{-1}B_2(s)^\top P_2(s)-P_2(s)B_2(s)R_{222}(s)^{-1}B_2(s)^\top P_1(s)=0,\quad s \in [t,T],
\end{aligned}
\end{equation}
with $P_1(T)=L_1$; and
\begin{equation}\label{Riccati eq-2}
\begin{aligned}
&\dot{P}_2(s)-\bigg[\big(B_1(s)^\top-R_{212}(s)R_{222}(s)^{-1}B_2(s)^\top\big)P_1(s)-R_{112}(s)R_{222}(s)^{-1}B_2(s)^\top P_2(s)\bigg]^\top\\
&\quad\times\bigg[R_{212}(s)R_{222}(s)^{-1}R_{112}(s)^\top+R_{112}(s)R_{222}(s)^{-1}R_{212}(s)^\top-R_{111}(s)\bigg]^{-1}\\
&\quad\times R_{212}(s)R_{222}(s)^{-1}R_{212}(s)^\top\\
&\quad\times\bigg[R_{212}(s)R_{222}(s)^{-1}R_{112}(s)^\top+R_{112}(s)R_{222}(s)^{-1}R_{212}(s)^\top-R_{111}(s)\bigg]^{-1}\\
&\quad\times\bigg[\big(B_1(s)^\top-R_{212}(s)R_{222}(s)^{-1}B_2(s)^\top\big)P_1(s)-R_{112}(s)R_{222}(s)^{-1}B_2(s)^\top P_2(s)\bigg]\\
&+P_2(s)\bigg[B_1(s)^\top-R_{212}(s)R_{222}(s)^{-1}B_2(s)^\top\bigg]^\top\\
&\quad\times\bigg[R_{212}(s)R_{222}(s)^{-1}R_{112}(s)^\top+R_{112}(s)R_{222}(s)^{-1}R_{212}(s)^\top-R_{111}(s)\bigg]^{-1}\\
&\quad\times\bigg[\big(B_1(s)^\top-R_{212}(s)R_{222}(s)^{-1}B_2(s)^\top\big)P_1(s)-R_{112}(s)R_{222}(s)^{-1}B_2(s)^\top P_2(s)\bigg]\\
&+\bigg[\big(B_1(s)^\top-R_{212}(s)R_{222}(s)^{-1}B_2(s)^\top\big)P_1(s)-R_{112}(s)R_{222}(s)^{-1}B_2(s)^\top P_2(s)\bigg]^\top\\
&\quad\times\bigg[R_{212}(s)R_{222}(s)^{-1}R_{112}(s)^\top+R_{112}(s)R_{222}(s)^{-1}R_{212}(s)^\top-R_{111}(s)\bigg]^{-1}\\
&\quad\times\bigg[B_1(s)^\top-R_{212}(s)R_{222}(s)^{-1}B_2(s)^\top\bigg]P_2(s)+C(s)^\top P_2(s)C(s)+Q_2(s)\\
&+P_2(s)A(s)+A(s)^\top P_2(s)-P_2(s)B_2(s)R_{222}(s)^{-1}B_2(s)^\top P_2(s)=0,\quad s \in [t,T],
\end{aligned}
\end{equation}
with $P_2(T)=L_2$.

Finally, the feedback Stackelberg equilibrium in this case is
\begin{equation*}
\begin{aligned}
&u^*(s,x)=T_1(s,x,P_1(s)x,P_2(s)x,P_1(s),P_2(s))\\
&=\bigg[R_{212}(s)R_{222}(s)^{-1}R_{112}(s)^\top+R_{112}(s)R_{222}(s)^{-1}R_{212}(s)^\top-R_{111}(s)\bigg]^{-1}\\
&\quad\times\bigg[B_1(s)^\top P_1(s)-R_{212}(s)R_{222}(s)^{-1}B_2(s)^\top P_1(s)-R_{112}(s)R_{222}(s)^{-1}B_2(s)^\top P_2(s)\bigg]x,\\
&v^*(s,x,u^*(s,x))=T_2(s,x,u^*(s,x),P_2(s)x,P_2(s))\\
&=-R_{222}(s)^{-1}R_{212}(s)^\top u^*(s,x)-R_{222}(s)^{-1}B_2(s)^\top P_2(s)x\\
&=-R_{222}(s)^{-1}R_{212}(s)^\top \bigg[R_{212}(s)R_{222}(s)^{-1}R_{112}(s)^\top+R_{112}(s)R_{222}(s)^{-1}R_{212}(s)^\top\\
&\qquad-R_{111}(s)\bigg]^{-1}\bigg[B_1(s)^\top P_1(s)-R_{212}(s)R_{222}(s)^{-1}B_2(s)^\top P_1(s)\\
&\qquad-R_{112}(s)R_{222}(s)^{-1}B_2(s)^\top P_2(s)\bigg]x-R_{222}(s)^{-1}B_2(s)^\top P_2(s)x,\quad s \in [t,T].
\end{aligned}
\end{equation*}

So far, we have not been able to obtain the solvability of Riccati equations $(\ref{Riccati eq-1})$ and $(\ref{Riccati eq-2})$. However, a special case can be solved. Taking $B_1(\cdot)=0, R_{212}(\cdot)=0$ in $(\ref{i})$ and $(\ref{j})$. In this case, there is only follower's control $v(\cdot)$ in the drift term of $(\ref{i})$. And there is a cross term in the leader's cost functional (In this situation feedback Nash equilibria and feedback Stackelberg equilibria are different, see \cite{BH1984}). The two Riccati equations $(\ref{Riccati eq-1})$ and $(\ref{Riccati eq-2})$ are reduced to:
\begin{equation}\label{Riccati eq-3}
\left\{
\begin{aligned}
&\dot{P}_1(s)+C(s)^\top P_1(s)C(s)+Q_1(s)+P_1(s)A(s)+A(s)^\top P_1(s)\\
&-P_1(s)B_2(s)R_{222}(s)^{-1}B_2(s)^\top P_2(s)-P_2(s)B_2(s)R_{222}(s)^{-1}B_2(s)^\top P_1(s),\\
&-P_2(s)B_2(s)R_{222}(s)^{-1}R_{112}(s)^\top R_{111}(s)^{-1}R_{112}(s)R_{222}(s)^{-1}B_2(s)^\top P_2(s)=0,\\
&R_{111}(s)>0,\quad s \in [t,T],\\
&P_1(T)=L_1,
\end{aligned}
\right.
\end{equation}
\begin{equation}\label{Riccati eq-4}
\left\{
\begin{aligned}
&\dot{P}_2(s)+C(s)^\top P_2(s)C(s)+Q_2(s)+P_2(s)A(s)+A(s)^\top P_2(s)\\
&-P_2(s)B_2(s)R_{222}(s)^{-1}B_2(s)^\top P_2(s)=0,\\
&R_{222}(s)>0,\quad s \in [t,T],\\
&P_2(T)=L_2.
\end{aligned}
\right.
\end{equation}
According to Theorem 7.2, Chapter 6 of Yong and Zhou \cite{yong1999stochastic}, we know that if $R_{222}(s)\gg 0, Q_2(s)\ge 0,s\in [t,T],L_2\ge 0$ and $R_{222}\in C([t,T];\mathcal{S}^{m_2})$, then the Riccati equation $(\ref{Riccati eq-4})$ admits a unique solution over $[t,T]$. And $R_{222}(s)\gg 0$ means $R_{222}(s)\ge \delta I, a.e. s \in [t,T]$ for some $\delta>0.$ Let $R_{111}(s)>0, s\in[t,T],$ and $(\ref{Riccati eq-4})$ admits a unique solution $P_2(\cdot)\in C([t,T];\mathcal{S}^{n})$. Since $(\ref{Riccati eq-3})$ is a linear {\it ordinary differential equation} (ODE, in short) with bounded coefficients, it follows that it has a unique solution $P_1(\cdot)$.

\begin{myrem}
Since the analytical solutions to the two Riccati equations $(\ref{Riccati eq-1})$ and $(\ref{Riccati eq-2})$ are difficult to discuss in general, we further discuss their numerical solutions with the certain particular coefficients. We are only considering one dimensional case here.

Let $A(s)=1, B_1(s)=1, B_2(s)=0.001, C(s)=1, Q_1(s)=1, Q_2(s)=1, R_{111}(s)=1, R_{112}(s)=1, R_{212}(s)=0.001, R_{222}(s)=1$, for $s\in[t,T]$, $L_1=1, L_2=2$, and $t=0,T=1$. Then $(\ref{Riccati eq-1})$ and $(\ref{Riccati eq-2})$ become
\begin{equation}\label{Riccati eq-5}
\left\{
\begin{aligned}
&\dot{P}_1(s)={0.998}^{-1}\big(0.999999P_1(s)-0.001P_2(s)\big)^2 +2\times 10^{-6}P_1(s)P_2(s)\\
&\qquad\quad -3P_1(s)-1,\quad s \in [0,1],\\
&P_1(1)=1,
\end{aligned}
\right.
\end{equation}
\begin{equation}\label{Riccati eq-6}
\left\{
\begin{aligned}
&\dot{P}_2(s)={0.998}^{-2}\times 10^{-6}(0.999999P_1(s)-0.001P_2(s))^2\\
&\quad\qquad +2\times 0.999999\times {0.998}^{-1}P_2(s) (0.999999P_1(s)-0.001P_2(s))\\
&\quad\qquad+10^{-6}P_2(s)^2-3P_2(s)-1,\quad s \in [0,1],\\
&P_2(1)=2,
\end{aligned}
\right.
\end{equation}
respectively. By making the time reversing transformation
\[r=1-s,\quad s \in [0,1],\]
$(\ref{Riccati eq-5})$ and $(\ref{Riccati eq-6})$ are equivalent to
\begin{equation}\label{Riccati eq-7}
\left\{
\begin{aligned}
&\dot{P}_1(r)=-\big[{0.998}^{-1}(0.999999P_1(r)-0.001P_2(r))^2 +2\times 10^{-6}P_1(r)P_2(r)\\
&\qquad\qquad -3P_1(r)-1\big],\quad r \in [0,1],\\
&P_1(0)=1,
\end{aligned}
\right.
\end{equation}
\begin{equation}\label{Riccati eq-8}
\left\{
\begin{aligned}
&\dot{P}_2(r)=-\big[{0.998}^{-2}\times 10^{-6}\big(0.999999P_1(r)-0.001P_2(r)\big)^2\\
&\quad\qquad\quad +2\times 0.999999\times {0.998}^{-1}P_2(r)\big(0.999999P_1(r)-0.001P_2(r)\big)\\
&\quad\qquad\quad +10^{-6}P_2(r)^2-3P_2(r)-1\big],\quad r \in [0,1],\\
&P_2(0)=2,
\end{aligned}
\right.
\end{equation}
respectively. We give some numerical simulation and plot a figure.
\begin{figure}[H]
  \centering
  \includegraphics[height=9cm,width=14cm]{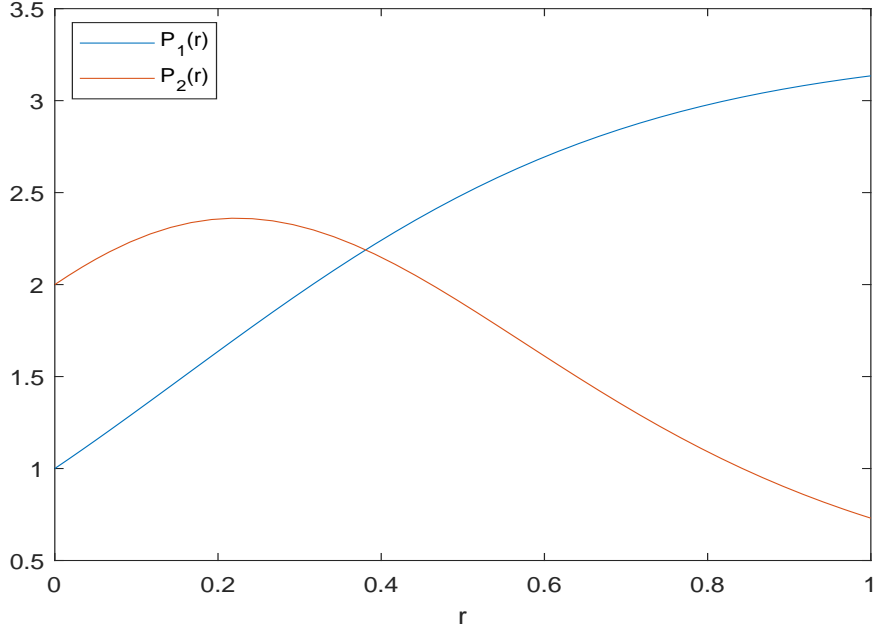}
  \caption{the numerical solutions to Riccati equations (\ref{Riccati eq-7}) and (\ref{Riccati eq-8})}
\end{figure}
\end{myrem}

\subsection{Case 2}\label{Case 2}
Let $b(\cdot)=0,B_2(\cdot)=0,C(\cdot)=0,D_1(\cdot)=0,\lambda(\cdot)=0,M_{11}(\cdot)=0,M_{12}(\cdot)=0,M_{21}(\cdot)=0,M_{22}(\cdot)=0,R_{122}(\cdot)=0,R_{211}(\cdot)=0,R_{212}(\cdot)=0,q_1(\cdot)=0,
\rho_{11}(\cdot)=0,\rho_{12}(\cdot)=0,q_2(\cdot)=0,\rho_{21}(\cdot)=0,\rho_{22}(\cdot)=0,N_1=0,N_2=0$ in $(\ref{LQ state SDE-1})$ and $(\ref{LQ cost functionals-leader and follower})$. Therefore, the state equation and cost functionals are of the following form:
\begin{equation}
\left\{
\begin{aligned}
&dx(s)=\big[A(s)x(s)+B_1(s)u(s)\big]ds+D_2(s)v(s)dW(s),\quad s \in [t,T],\\
&x(t)=x,
\end{aligned}
\right.
\end{equation}
and
\begin{equation}
\begin{aligned}
J_1(t,x;u(\cdot),v(\cdot))=\frac{1}{2}\mathbb{E}&\bigg\{\int_t^T\big[x(s)^\top Q_1(s)x(s)+u(s)^\top R_{111}(s)u(s)+2u(s)^\top R_{112}(s)v(s)\big]ds\\
&\quad+\langle L_1 x(T),x(T)\rangle\bigg\},\\
J_2(t,x;u(\cdot),v(\cdot))=\frac{1}{2}\mathbb{E}&\bigg\{\int_t^T\big[x(s)^\top Q_2(s)x(s)+v(s)^\top R_{222}(s)v(s)\big]ds+\langle L_2 x(T),x(T)\rangle\bigg\}.
\end{aligned}
\end{equation}
Following the calculation steps as in Case $\ref{Case 1}$, by assuming
\begin{equation*}
\left\{
\begin{aligned}
&R_{111}(s)>0,\\
&D_2(s)^\top P_2(s)D_2(s)+R_{222}(s)>0,\quad s \in [t,T],
\end{aligned}
\right.
\end{equation*}
we obtain the corresponding Riccati equations:
\begin{equation}\label{Riccati eq-9}
\left\{
\begin{aligned}
&\dot{P}_1(s)+P_1(s)A(s)+A(s)^\top P_1(s)+Q_1(s)\\
&\quad -P_1(s)B_1(s)R_{111}(s)^{-1}B_1(s)^\top P_1(s)=0,\quad s\in[t,T],\\
&P_1(T)=L_1,
\end{aligned}
\right.
\end{equation}
\begin{equation}\label{Riccati eq-10}
\left\{
\begin{aligned}
&\dot{P}_2(s)+P_2(s)A(s)+A(s)^\top P_2(s)+Q_2(s)-P_2(s)B_1(s)R_{111}(s)^{-1}B_1(s)^\top P_1(s)\\
&\qquad-P_1(s)B_1(s)R_{111}(s)^{-1}B_1(s)^\top P_2(s)=0,\quad s\in[t,T],\\
&P_2(T)=L_2.
\end{aligned}
\right.
\end{equation}
According again Theorem 7.2, Chapter 6 of \cite{yong1999stochastic}, we know that if $R_{111}(s)\gg 0, Q_1(s)\ge 0,s\in [t,T],L_1\ge 0$ and $R_{111}\in C([t,T];\mathcal{S}^{m_1})$, then the Riccati equation $(\ref{Riccati eq-9})$ admits a unique solution over $[t,T]$. In the same way, if $D_2\in C([t,T];\mathbb{R}^{n\times m_2}),R_{222}\in C([t,T];\mathcal{S}^{m_2}),R_{222}(s)\gg 0, Q_2(s)\ge 0,s\in [t,T],L_2\ge 0$ and $(\ref{Riccati eq-9})$ admits a unique solution $P_1(\cdot)\in C([t,T];\mathcal{S}^{n})$, then the Riccati equation $(\ref{Riccati eq-10})$ admits a unique solution over $[t,T]$.

Finally, the feedback Stackelberg equilibrium in this case is
\begin{equation*}
\begin{aligned}
u^*(s,x)&=T_1(s,x,P_1(s)x,P_2(s)x,P_1(s),P_2(s))\\
        &=-R_{111}(s)^{-1}B_1(s)^TP_1(s)x,\\
v^*(s,x,u^*(s,x))&=T_2(s,x,u^*(s,x),P_2(s)x,P_2(s))=0,\quad s \in [t,T].
\end{aligned}
\end{equation*}

\section{Concluding Remarks}\label{Section 5}

Different from \cite{BCS2014}, we consider a finite-horizon Stackelberg stochastic differential game where both the drift term and diffusion term of the state equation contain the leader's and the follower's control variables. Due to the finite horizon feature, the verification theorem of the feedback Stackelberg equilibrium consists of parabolic PDEs. An LQ problem is further researched. We obtained the representation of the feedback Stackelberg equilibrium in two special cases, via related Riccati equations. We discuss the analytical and numerical solutions to the Riccati equations in some special cases. The general solvability of the corresponding Riccati equations requires systematic study. We will consider this topic in the future research.


\end{document}